\documentclass{amsart}
\usepackage{graphics,amsmath,amsthm,amssymb}
\usepackage{color}
\newtheorem{theorem}{Theorem}[section]

\newtheorem{remark}[theorem]{Remark}
\newtheorem{lemma}[theorem]{Lemma}

\begin{document}

\title[Congruences for the Fishburn Numbers]{Congruences for the Fishburn Numbers}
\author[G. E. Andrews]{George E. Andrews}
\address{Department of Mathematics, Penn State University, University Park, PA  16802, USA, gea1@psu.edu}
\author[J. A. Sellers]{James A. Sellers}
\address{Department of Mathematics, Penn State University, University Park, PA  16802, USA, sellersj@psu.edu}
\thanks{George E. Andrews was partially supported by NSA grant H98230--12--1--0205}

\date{\today}

\begin{abstract}
The Fishburn numbers, $\xi(n),$ are defined by a formal power series expansion
$$
\sum_{n=0}^\infty \xi(n)q^n = 1 + \sum_{n=1}^\infty \prod_{j=1}^n (1-(1-q)^j).
$$
For half of the primes $p$, there is a non--empty set of numbers $T(p)$ lying in $[0,p-1]$ such that if $j\in T(p),$ then for all $n\geq 0,$ 
$$
\xi(pn+j)\equiv 0 \pmod{p}.
$$
\end{abstract}

\maketitle

% \bigskip

\bigskip


\section{Introduction}
\label{intro}
The Fishburn numbers $\xi(n)$ are defined by the formal power series 
\begin{equation}
\label{Fishburn_genfn1}
\sum_{n=0}^\infty \xi(n)q^n = \sum_{n=0}^\infty (1-q;1-q)_n
\end{equation}
where 
\begin{equation}
\label{qq_notation}
(A;q)_n = (1-A)(1-Aq)\dots (1-Aq^{n-1}).
\end{equation}

The Fishburn numbers have arisen in a wide variety of combinatorial settings.  One can gain some sense of the extent of their applications in \cite[Sequence A022493]{Sl}.  Namely, these numbers arise in such combinatorial settings as linearized chord diagrams, Stoimenow diagrams, nonisomorphic interval orders, unlabeled $(2+2)$-free posets, and ascent sequences.  They were first defined in the work of Fishburn (cf. \cite{F1, F2, F3}), and have recently found a connection with mock modular forms \cite{BOPR}.  

It turns out that the Fishburn numbers satisfy congruences reminiscent of those for the partition function $p(n)$ \cite[Chapter 1]{A}.  Surprisingly, in contrast to $p(n),$ we shall see in Section \ref{identify_primes} that there are congruences of the form $\xi(pn+b)\equiv 0 \pmod{p}$ for half of all the primes $p.$   For example, for all $n\geq 0,$  
\begin{align}
  \xi(5n+3) &\equiv \xi(5n+4)\equiv 0 \pmod{5},  \label{mod5congs} \\
  \xi(7n+6) &\equiv 0 \pmod{7},  \label{mod7congs} \\
  \xi(11+8) &\equiv \xi(11n+9)\equiv \xi(11n+10)\equiv 0 \pmod{11}, \label{mod11congs}\\
  \xi(17n+16) &\equiv 0 \pmod{17}, \text{\ \ and} \label{mod17congs} \\
  \xi(19n+17) &\equiv \xi(19n+18)\equiv 0 \pmod{19}. \label{mod19congs}
\end{align}
These results all follow from a general result stated as Theorem \ref{mainthm} in Section \ref{MainThm}.   The next section is devoted to background lemmas.  Theorem \ref{mainthm} is then proved in Section  \ref{MainThm}.  In Section \ref{identify_primes} we discuss an infinite family of primes $p$ for which these congruences hold.  We conclude with some open problems.  

\section{Background Lemmas}
\label{background}

The sequence of pentagonal numbers is given by 
\begin{equation}
\label{pentagonal_numbers}
\left\{  n(3n-1)/2 \right\}_{n=-\infty}^\infty = \left\{0,1,2,5,7,12,15,22,\dots\right\}.
\end{equation}
Throughout this work the symbol $\lambda$ will be used to designate a pentagonal number.  

In our first lemma, $f(q)$ will denote an arbitrary polynomial in ${\Bbb Z}[q],$ and $p$ will be a fixed prime.  Then we separate the terms in $f(q)$ according to the residue of the exponent modulo $p.$   Thus, 
\begin{equation}
\label{poly_dissection1}
f(q) = \sum_{i=0}^{p-1} q^i\phi_i(q^p).
\end{equation}

We also suppose that for every $p^{th}$ root of unity $\zeta$ (including $\zeta=1$), 
$$f(\zeta) = \sum_{\lambda} c_{\lambda}\zeta^\lambda$$
where the $\lambda$'s sum over some set of pentagonal numbers that includes 0.  The $c$'s are thus defined to be 0 outside this prescribed set of pentagonal numbers, and the $c$'s are independent of the choice of $\zeta.$  

\begin{lemma}
\label{lemma1}
Under the above conditions, $\phi_j(1) = 0$ if $j$ is not a pentagonal number.
\end{lemma}

\begin{proof}
The assertion is not immediate because the $p^{th}$ roots of unity are not linearly independent.  In particular, if $\zeta$ is a primitive $p^{th}$ root of unity, then 
$$
1+\zeta + \zeta^2 +\dots + \zeta^{p-1} = 0.
$$
However, we know that the ring of integers in ${\Bbb Q}(\zeta)$ has $1, \zeta, \zeta^2, \dots, \zeta^{p-2}$ as a basis \cite[page 187]{AW}.  Hence, 
$$
\phi_0(1)(-\zeta - \zeta^2-\dots - \zeta^{p-1}) + \sum_{j=1}^{p-1}\zeta^j\phi_j(1) = c_0(-\zeta - \zeta^2-\dots - \zeta^{p-1})+\sum_{\lambda\neq 0}c_\lambda \zeta^\lambda.
$$
Therefore, if $1\leq j\leq p-1,$
$$
\phi_j(1) - \phi_0(1) = 
\begin{cases}
c_\lambda - c_0 & \text{if $j$ is one of the designated pentagonal numbers},\\
-c_0 & \text{otherwise}
\end{cases}
$$
is a linear system of $p-1$ equations in $p$ variables $\phi_j(1),$ $0\leq j\leq p-1.$   However, the $\zeta = 1$ case adds one further equation 
$$
\phi_0(1) + \phi_1(1) + \dots + \phi_{p-1}(1) = \sum_\lambda c_\lambda.
$$
We now have a linear system of $p$ equations in $p$ variables, and the determinant of the system is $p.$  Hence, there is a unique solution which is the obvious solution
\begin{equation*}
\phi_j(1)=
\begin{cases}
c_\lambda & \text{if $j$ is one of the designated pentagonal numbers},\\
0 & \text{otherwise}.
\end{cases}
\end{equation*}
\end{proof}

In the next three lemmas, we require some variations on Leibniz's rule for taking the $n^{th}$ derivative of a product.  Each is probably in the literature, but is included here for completeness. 

\begin{lemma}
\label{lemma2}
$$
\left(q\frac{d}{dq}\right)^n (A(q)B(q)) = \sum_{j=1}^n q^jc_{n,j}\left(\frac{d}{dq}\right)^j(A(q)B(q)),
$$
where the $c_{n,j}$ are the Stirling numbers of the second kind given by $c_{n,0} = c_{n,n+1} = 0,$ $c_{1,1}=1,$ and $c_{n+1,j} = jc_{n,j}+c_{n,j-1}$ for $1\leq j\leq n+1.$   
\end{lemma}

\begin{proof}
The result is a tautology when $n=1.$  To pass from $n$ to $n+1,$ we note 
\begin{eqnarray*}
\left(q\frac{d}{dq}\right)^{n+1} (A(q)B(q))
&=& 
q\frac{d}{dq}\left( \left(q\frac{d}{dq}\right)^n (A(q)B(q)) \right) \\
&=& 
q\frac{d}{dq}\sum_{j=1}^n q^jc_{n,j}\left(\frac{d}{dq}\right)^j(A(q)B(q))\\
&=& 
q\frac{d}{dq} \sum_{j=1}^nq^jc_{n,j}\left(\frac{d}{dq}\right)^j(A(q)B(q)) \\
&=& 
\sum_{j=1}^n jq^jc_{n,j}\left(\frac{d}{dq}\right)^j(A(q)B(q)) \\
&& 
\qquad
+ \sum_{j=1}^n q^{j+1}c_{n,j}\left(\frac{d}{dq}\right)^{j+1}(A(q)B(q))\\
&=& 
\sum_{j=1}^{n+1} q^j(jc_{n,j}+c_{n,j-1})\left(\frac{d}{dq}\right)^j(A(q)B(q)) \\
&=& 
\sum_{j=1}^{n+1} q^jc_{n+1,j}\left(\frac{d}{dq}\right)^j(A(q)B(q)).
\end{eqnarray*}
\end{proof}

\begin{lemma}
\label{lemma3}
\begin{equation*}
\left(\frac{d}{dt}\right)^n f(qe^t)\biggr\rvert_{t=0} = \left(q\frac{d}{dq}\right)^n f(q).  
\end{equation*}
\end{lemma}

\begin{proof}
By Lemma \ref{lemma2} with $A(q) = f(q)$ and $B(q) = 1,$ we see that 
\begin{equation}
\label{eqn2.3}
\left(q\frac{d}{dq}\right)^n f(q) = \sum_{j=1}^n q^jc_{n,i}f^{(j)}(q).
\end{equation}
On the other hand, we claim 
\begin{equation}
\label{eqn2.4}
\left(\frac{d}{dt}\right)^n f(qe^t) = \sum_{j=1}^n q^je^{jt}c_{n,j}f^{(j)}(qe^t).
\end{equation}
When $n=1,$ this is just the chain rule applied to $f(qe^t).$  To pass from $n$ to $n+1,$ we note 
\begin{eqnarray*}
\left(\frac{d}{dt}\right)^{n+1} f(qe^t) 
&=& 
\frac{d}{dt}\left(\frac{d}{dt}\right)^{n} f(qe^t) \\
&=& 
\frac{d}{dt} \sum_{j=1}^n q^je^{jt}c_{n,j}f^{(j)}(qe^t) \\
&=& 
\sum_{j=1}^n j q^je^{jt}c_{n,j}f^{(j)}(qe^t) \\
&&
\qquad 
+\sum_{j=1}^n q^{j+1}e^{(j+1)t}c_{n,j}f^{(j+1)}(qe^t) \\
&=& 
\sum_{j=1}^{n+1} (jc_{n,j}+c_{n,j-1})q^je^{jt}f^{(j)}(qe^t) \\
&=& 
\sum_{j=1}^{n+1} c_{n+1,j}q^je^{jt}f^{(j)}(qe^t).
\end{eqnarray*}
Comparing (\ref{eqn2.4}) with $t=0$ to (\ref{eqn2.3}), we see that our lemma is established.  
\end{proof}
We now turn to the generating function for the Fishburn numbers as given by Zagier \cite[page 946]{Z}.  Namely, 
\begin{equation}
\label{eqn2.5}
F(1-q) = \sum_{n=0}^\infty \xi(n)q^n = \sum_{n=0}^\infty (1-q; 1-q)_n.
\end{equation}
To facilitate the study, we concentrate on 
\begin{equation}
\label{eqn2.6}
F(q) = \sum_{n=0}^\infty (q;q)_n
\end{equation}
and 
\begin{equation}
\label{eqn2.7}
F(q,N) = \sum_{n=0}^N (q;q)_n = \sum_{i=0}^{p-1} q^i A_p(N,i, q^p),
\end{equation}
where $A_p(N,i,q^p)$ is a polynomial in $q^p.$  
We note that if $\zeta$ is a $p^{th}$ root of unity 
\begin{equation}
\label{eqn2.8}
F(\zeta) = F(\zeta, m) = F(\zeta, p-1)
\end{equation}
for all $m\geq p.$  Furthermore, 
\begin{equation}
\label{eqn2.9}
\left(q\frac{d}{dq}\right)^r F(q)\biggr\rvert_{q=\zeta} = \left(q\frac{d}{dq}\right)^r F(q, m)\biggr\rvert_{q=\zeta}  = \left(q\frac{d}{dq}\right)^r F(q, (r+1)p-1)\biggr\rvert_{q=\zeta}
\end{equation}
for all $m\geq (r+1)p$ because $(1-q^p)^{r+1}$ divides $(q;q)_j$ for all $j\geq (r+1)p.$  

Similarly, for all $m\geq (r+1)p,$ 
\begin{equation}
\label{eqn2.10}
F^{(r)}(q)\biggr\rvert_{q=\zeta} = F^{(r)}(q,m)\biggr\rvert_{q=\zeta} = F^{(r)}(q,(r+1)p-1)\biggr\rvert_{q=\zeta}.
\end{equation}
In the next lemma, we require a Stirling--like array of numbers $C_{N,i,j}(p)$ given by $C_{N,i,0}(p) = i^N$ $(C_{0,0,0}(p) = 1)$, $C_{N,i,N+1}(p) = 0,$ and for $1\leq j\leq N,$ 
\begin{equation}
\label{eqn2.11}
C_{N+1,i,j}(p) = (i+jp)C_{N,i,j}(p)+pC_{N,i,j-1}(p).
\end{equation}

\begin{lemma}
\label{lemma4}
$$
\left(q\frac{d}{dq}\right)^NF(q,n) = \sum_{j=0}^N \sum_{i=0}^{p-1} C_{N,i,j}(p)q^{i+jp}A_p^{(j)}(n,i,q^p).
$$
\end{lemma}

\begin{proof}
In light of the fact that $C_{0,i,0}(p) = 1$ for all $i,$ the $N=0$ assertion is 
$$
F(q,n) = \sum_{i=0}^{p-1} q^i A_p(n,i,q^p),
$$
which is just the definition of the $A$'s given in (\ref{eqn2.7}).  To pass from $N$ to $N+1,$ we note 
\begin{eqnarray*}
\left(q\frac{d}{dq}\right)^{N+1} F(q,n) 
&=& 
q\frac{d}{dq}\sum_{j=0}^N\sum_{i=0}^{p-1}C_{N,i,j}(p)q^{i+jp}A_p^{(j)}(n,i,q^p) \\
&=& 
\sum_{j=0}^N\sum_{i=0}^{p-1}C_{N,i,j}(p)(i+jp)q^{i+jp}A_p^{(j)}(n,i,q^p) \\
&&
\qquad 
+ \sum_{j=0}^N\sum_{i=0}^{p-1}C_{N,i,j}(p)q^{i+jp}pq^pA_p^{(j+1)}(n,i,q^p) \\
&=& 
\sum_{j=0}^{N+1}\sum_{i=0}^{p-1}((i+jp)C_{N,i,j}(p) + pC_{N,i,j-1}(p)) q^{i+jp}A_p^{(j)}(n,i,q^p) \\
&=& 
\sum_{j=0}^{N+1}\sum_{i=0}^{p-1}C_{N+1,i,j}(p) q^{i+jp}A_p^{(j)}(n,i,q^p).
\end{eqnarray*}
\end{proof}

We now define, for any positive integer $p,$ two special sets of integers:  

\begin{equation}
\label{eqn2.12}
% S(p) \text{\ is the set of least nonnegative residues of all the pentagonal numbers modulo $p$},
S(p) = \left\{ j \, \vert\,  0\leq j\leq p-1 \text{\ such that\ } n(3n-1)/2 \equiv j \pmod{p} \text{\ for some } n   \right\}
\end{equation}
and 
\begin{equation}
\label{eqn2.13}
% T(p) \text{\ is the set of all integers in $[0,p-1]$ that are each larger than every element of $S(p).$}  
T(p) = \left\{ k \, \vert\, 0\leq k\leq p-1 \text{\ such that\ } k \text{\ is larger than every element of $S(p)$}    \right\}.
\end{equation}
For example, for $p=11,$ we have 
$$
S(11) = \left\{0,1,2,4,5,7\right\} \text{\ \ \ and \ \ \ } T(11) = \left\{8,9,10\right\}.
$$

\begin{lemma}
\label{lemma5}
If $i \not\in S(p),$ then 
$$
A_p(pn-1,i,q) = (1-q)^n\alpha_p(n,i,q)
$$
where the $\alpha_p(n,i,q)$ are polynomials in ${\Bbb Z}[q].$  
\end{lemma}

\begin{proof}
This result is equivalent to the assertion that for $0\leq j < n,$ 
$$
A_p^{(j)}(pn-1,i,1) = 0,
$$
and by (\ref{eqn2.10}) we need only prove for $j\geq 0,$ 
\begin{equation}
\label{eqn2.14}
A_p^{(j)}((j+1)p-1,i,1) = 0
\end{equation}
because $n\geq (j+1).$  

We proceed to prove (\ref{eqn2.14}) by induction on $j.$  When $j=0,$ we only need show that if $i\not\in S(p),$ 
$$A_p(p-1,i,1)=0.$$
Following \cite[Section 5]{Z}, we define (where $\zeta$ is now an $N^{th}$ root of unity)
\begin{align}
  F(\zeta e^t) &= \sum_{n=0}^\infty \frac{b_n(\zeta)t^n}{n!} \label{eqn2.15} \\
    &= e^{t/24}\sum_{n=0}^\infty \frac{c_n(\zeta)t^n}{24^nn!} \nonumber \\
    & = \sum_{M=0}^\infty \frac{t^M}{24^MM!}\sum_{n=0}^M \binom{M}{n}c_n(\zeta), \nonumber
\end{align}
where we have replaced Zagier's $\xi$ with $\zeta$ to avoid confusion with $\xi(n).$  
In \cite[Section 5]{Z}, we see that 
\begin{equation}
\label{eqn2.16}
c_n(\zeta) = \frac{(-1)^nN^{2n+1}}{2n+2}\sum_{m=1}^{N/2} \chi(m)\zeta^{(m^2-1)/24}B_{2n+2}\left( \frac{m}{N} \right),
\end{equation}
where the $B$'s are Bernoulli polynomials and $\chi(m) = \left( \frac{12}{m} \right).$  Note that the only non--zero terms in the sum in (\ref{eqn2.16}) have 
\begin{equation}
\label{eqn2.17}
\zeta^{((6m\pm 1)^2-1)/24}\chi(6m\pm 1) = (-1)^m\zeta^{m(3m\pm 1)/2},
\end{equation}
i.e., $c_n(\zeta)$ is a linear combination of powers of $\zeta$ where each exponent is a pentagonal number.  Hence, by (\ref{eqn2.15}) we see that $b_n(\zeta)$ is a linear combination of powers of $\zeta$ where each exponent is a pentagonal number.  

Hence, if $\zeta$ is now a $p^{th}$ root of unity, 
\begin{eqnarray*}
F(\zeta) 
&=& 
F(\zeta, p-1) \\
&=& 
b_0(\zeta) \\
&=& \sum_\lambda c_\lambda \zeta^\lambda,
\end{eqnarray*}
where the sum over $\lambda$ is restricted to a subset of the pentagonal numbers. On the other hand, 
\begin{eqnarray*}
F(\zeta)
&=& 
F(\zeta, p-1) \\
&=& 
\sum_{i=0}^{p-1} \zeta^iA_p(p-1,i,1).
\end{eqnarray*}
Hence, by Lemma \ref{lemma1}, for $i\not\in S(p),$ 
$$
A_p(p-1,i,1) = 0
$$
which is (\ref{eqn2.14}) when $j=0.$  
Now let us assume that 
\begin{equation}
\label{eqn2.18}
A_p^{(j)}(p(j+1)-1,i,1) = 0
\end{equation}
for $0\leq j < \nu < n.$  
By Lemma \ref{lemma4}, 
\begin{equation}
\label{eqn2.19}
\left(q\frac{d}{dq}\right)^\nu F(q,p(\nu+1)-1) = \sum_{j=0}^\nu \sum_{i=0}^{p-1} C_{\nu, i, j}(p)\zeta^iA_p^{(j)}(p(\nu +1)-1,i,1).
\end{equation}
But for $j < \nu,$ 
$$
A_p^{(j)}(p(\nu + 1), i, 1) = A_p^{(j)}(p(j+1)-1,i,1) = 0.
$$
Hence the only terms in the sum in (\ref{eqn2.19}) where $\zeta$ is raised to a non--pentagonal power, $i,$ arise from the terms with $j=\nu,$ namely 
\begin{equation}
\label{eqn2.20}
C_{\nu, i, \nu}(p)\zeta^iA_p^{(\nu)}(p(\nu +1)-1,i,1),
\end{equation}
and we note that $C_{\nu,i,\nu}(p) \neq 0.$  

Applying Lemma \ref{lemma3} to the left side of (\ref{eqn2.19}), we see that by (\ref{eqn2.15}) 
\begin{align}
b_{\nu}(\zeta)  &= \left(  q\frac{d}{dq}\right)^\nu F(q)\biggr\rvert_{q=\zeta} \label{eqn2.21} \\
    &= \left(  q\frac{d}{dq}\right)^\nu F(q, (\nu + 1)p-1)\nonumber \\
    & = \sum_{j=0}^\nu \sum_{i=0}^{p-1} C_{\nu,i,j}(p)\zeta^iA_p^{(j)}(p(\nu + 1)-1,i,1).  \nonumber
\end{align}
Recall that $b_\nu(\zeta)$ is a linear combination of powers of $\zeta$ where the exponents are pentagonal numbers.  Hence the expression given in (\ref{eqn2.20}) must be zero by Lemma \ref{lemma1}.  Therefore, 
$$
A_p^{(\nu )}(p(\nu + 1)-1,i,1) = 0,
$$
and this proves (\ref{eqn2.14}) and thus proves Lemma \ref{lemma5}.  
\end{proof}

\section{The Main Theorem}
\label{MainThm}
We recall from (\ref{eqn2.5}) that 
\begin{eqnarray*}
\sum_{n=0}^\infty \xi(n)q^n 
&=& 
\sum_{j=0}^\infty (1-q;1-q)_j \\
&=& 
1+\sum_{j=1}^\infty  \prod_{i=1}^j \sum_{h=1}^i (-1)^{h-1}q^h\binom{i}{h} \\
&=& 
1+\sum_{j=1}^\infty (q^j + O(q^{j+1})).
\end{eqnarray*}
Hence, 
\begin{equation}
\label{eqn3.1}
\sum_{n=0}^\infty \xi(n)q^n = F(1-q,N) + O(q^{N+1}).  
\end{equation}
We are now in a position to state and prove the main theorem of this paper.  
\begin{theorem}
\label{mainthm}
If $p$ is a prime and $i\in T(p)$ (as defined in (\ref{eqn2.13})), then for all $n\geq 0,$ 
$$\xi(pn+i) \equiv 0 \pmod{p}.$$
\end{theorem}
\begin{remark}
Congruences (\ref{mod5congs})--(\ref{mod19congs}) are the cases $p=5,7,11,17$ and $19$ of Theorem \ref{mainthm}.  
\end{remark}
\begin{proof}
We begin with a simple observation derived from Lucas's theorem for the congruence class of binomial coefficients modulo $p$ \cite[page 271]{D}.  Namely if $\pi$ is any integer congruent to a pentagonal number modulo $p,$ and $i\in T(p),$ then 
\begin{equation}
\label{eqn3.2}
\binom{\pi}{i} \equiv 0\pmod{p},
\end{equation}
because the final digit in the $p$--ary expansion of $\pi$ is smaller than $i$ because $i$ is in $T(p).$  

Now by Lemma \ref{lemma5}, we may write 
\begin{eqnarray*}
F(q,pn-1) 
&=& 
\sum_{i=0}^{p-1} q^i A_p(pn-1,i,q^p) \\
&=& 
\sum_{i=0 \atop i\in S(p)}^{p-1} q^i A_p(pn-1,i,q^p) + \sum_{i=0 \atop i\not\in S(p)}^{p-1} q^i (1-q^p)^n\alpha_p(n,i,q^p).
\end{eqnarray*}
So 
\begin{eqnarray*}
F(1-q,pn-1) 
&=& 
\sum_{i=0 \atop i\in S(p)}^{p-1} (1-q)^i A_p(pn-1,i,(1-q)^p) \\
&&
\qquad 
+ \sum_{i=0 \atop i\not\in S(p)}^{p-1} (1-q)^i (1-(1-q)^p)^n\alpha_p(n,i,(1-q)^p) \\
&:=& 
\Sigma_1 + \Sigma_2.
\end{eqnarray*}
Now modulo $p,$ 
\begin{eqnarray*}
\Sigma_2 
&\equiv & 
\sum_{i=0 \atop i\not\in S(p)}^{p-1} (1-q)^i q^{pn}\alpha_p(n,i,1) \\
&=& 
O(q^{pn}).
\end{eqnarray*}
Therefore, modulo $p,$ 
$$
F(1-q,pn-1) \equiv \sum_{i=0 \atop i\in S(p)}^{p-1} (1-q)^i A_p(pn-1,i,1-q^p) \pmod{p}.\\
$$
Let us look at the terms in this sum where $q$ is raised to a power that is congruent to an element of $T(p).$  Such a term must arise from the expansion of some $(1-q)^i$ where $i\in S(p)$ because $A_p(pn-1,1,1-q^p)$ is a polynomial in $q^p.$  

By (\ref{eqn3.2}) all such terms have a coefficient congruent to 0 modulo $p.$  Therefore, every term $q^j$ in $F(1-q,pn-1)$ where $j$ is congruent to an element of $T(p)$ must have a coefficient congruent to 0 modulo $p.$  

To conclude the proof, we let $n\to\infty.$  
\end{proof}

\section{An Infinite Set of Primes With Congruences}
\label{identify_primes} 
At this stage, one might ask whether one can identify an infinite set of primes $p$ for which congruences such as those described in Theorem \ref{mainthm} are found.  The answer to this question can be answered affirmatively.  

\begin{theorem}
Let $R = \{5,7,10,11,14,15,17,19,20,21,22\}.$  (The elements of $R$ are those numbers $r,$  $0 < r < 23,$ such that $\left( \frac{r}{23} \right) = -1.$)  Let $p$ be a prime of the form $p = 23k + r$ for some nonnegative integer $k$ and some $r \in R.$  Then $T(p)$ is not empty, i.e., at least one congruence such as those described in Theorem \ref{mainthm} must hold modulo $p.$  
\end{theorem}

\begin{remark}
From the Prime Number Theorem for primes in arithmetic progression, we see that, asymptotically, $T(p)$ is not empty for half of the primes and $T(p)$ equals the empty set for half of the primes.  
\end{remark}

\begin{proof}
Assume $p$ is a prime for which $T(p)$ {\bf is} empty.  That means there is a pentagonal number which is congruent to $- 1$ modulo $p.$  Then $n(3n-1)/2 \equiv -1 \pmod{p}$ for some integer $n.$  By completing the square we then obtain $(6n-1)^2 \equiv -23 \pmod{p}.$  Thus, by contrapositive, if we know that $-23$ is a quadratic nonresidue modulo $p,$ then we know that such a pentagonal number does not exist (which means $T(p)$ is not empty). 

Thus, if $\left( \frac{-23}{p} \right) = -1,$ then $T(p)$ is not empty.  But thanks to properties of the Legendre symbol, we know 
\begin{eqnarray*}
\left( \frac{-23}{p} \right)
&=&
\left( \frac{-1}{p} \right)\left( \frac{23}{p} \right) \\
&=& 
(-1)^{\frac{p-1}{2}}(-1)^{\frac{23-1}{2}\frac{p-1}{2}}\left( \frac{p}{23} \right) \text{\ \   by quadratic reciprocity} \\
&=& 
(-1)^{\frac{12(p-1)}{2}}\left( \frac{r}{23} \right)  \text{\ \ since $p = 23k + r$} \\
&=& 
\left( \frac{r}{23} \right)
\end{eqnarray*}
and we want this value to be $-1.$  The theorem then follows by the nature of the construction of $R.$  
\end{proof}
Thus, we clearly have infinitely many primes $p$ for which the Fishburn numbers will exhibit at least one congruence modulo $p.$  

\section{Conclusion}
There are many natural open questions that could be answered at this point.  

\begin{itemize}

\item{} First, we believe that Theorem \ref{mainthm} lists all the congruences of the form $\xi(pn+b) \equiv 0 \pmod{p},$ but we have not proved this at this time.  

\vskip .2in

\item{} Numerical evidence seems to indicate that Theorem \ref{mainthm} can be strengthened.  Namely, for certain values of $j > 1$ and certain primes $p,$ it appears that 
$$
\xi(p^j n + b) \equiv 0 \pmod{p^j}
$$
for certain values $b$ and all $n.$  

\vskip .2in

\item{}  Numerical evidence suggests that Lemma \ref{lemma5} could be strengthened as follows:  If $i\not\in S(p),$ then 
$$
A_p(pn-1,i,q) = (q;q)_n\beta_p(n,i,q)
$$
for some polynomial $\beta_p(n,i,q).$  That is to say, in Lemma 5, it was proved that $(1-q)^n$ divides $A_p(pn-1,i,q)$; it appears that the factor $(1-q)^n$ can be strengthened to $(q;q)_n.$

\vskip .2in

\item{}  With an eye towards the recent work of Andrews and Jel{\'i}nek \cite{AJ}, consider the power series given by 
\begin{eqnarray*}
\sum_{n=0}^\infty a(n)q^n 
&:=&
 \sum_{n=0}^\infty \left(\frac{1}{1-q},\frac{1}{1-q}\right)_n 
\end{eqnarray*}
which begins
\begin{equation*}
\qquad \qquad 1 - q + q^2 - 2 q^3 + 5 q^4 - 16 q^5 + 61 q^6 - 271 q^7 + 1372 q^8 - 7795 q^9 + \dots 
\end{equation*}
We conjecture that, for all $n\geq 0,$ $a(5n+4) \equiv 0 \pmod{5}.$  

\end{itemize}
  
\bibliographystyle{amsplain}

\begin{thebibliography}{99}

\bibitem{AW}  S. Alaca and K. S. Williams,  {\it Introductory Algebraic Number Theory}, Cambridge University Press, Cambridge, 2004

\bibitem{A} G. E. Andrews, {\it The Theory of Partitions}, Addison-Wesley, Reading 1976; reprinted, Cambridge University Press, Cambridge, 1984, 1998  

\bibitem{AJ} G. E. Andrews and V. Jel{\'i}nek, On $q$--Series Identities Related to Interval Orders, to appear in {\it European J. Combin.}

\bibitem{BOPR}  J. Bryson, K. Ono, S. Pitman, and R. C. Rhoades, Unimodal sequences and quantum and mock modular forms, 
{\it Proc. Natl. Acad. Sci. USA}  {\bf 109} no. 40 (2012), 16063--16067

\bibitem{D}    L. E. Dickson, {\it History of the Theory of Numbers, Vol. I}, Chelsea Publishing Co., New York, 1966, reprinted by Dover Publishing, New York, 2005

\bibitem{F1}  P. C. Fishburn, Intransitive indifference with unequal indifference intervals, {\it J. Mathematical Psychology} {\bf 7} (1970),  144--149

\bibitem{F2}  P. C. Fishburn, Intransitive indifference in preference theory: A survey, {\it Operations Res.} {\bf 18} (1970),  207--228

\bibitem{F3}   P. C. Fishburn, {\it Interval orders and interval graphs}, John Wiley \& Sons, New York, 1985 

\bibitem{Sl}  N. J. A. Sloane, The On-Line Encyclopedia of Integer Sequences, published electronically at http://oeis.org, 2014

% \bibitem{St}   A. Stoimenow,  Enumeration of chord diagrams and an upper bound for Vassiliev invariants,  {\it J. Knot Theory Ramifications} {\bf 7} (1998), 93--114

\bibitem{Z}  D. Zagier, Vassiliev invariants and a strange identity related to the Dedekind eta-function, {\it Topology} {\bf 40} no. 5 (2001), 945--960


\end{thebibliography}

\end{document}